\documentclass[reqno, 12pt]{amsart}

\usepackage{amsmath}
\usepackage{amsfonts}
\usepackage{amssymb}
\usepackage{mathrsfs}
\usepackage{amsthm}
\usepackage{mathtools}
\usepackage{bbm}
\usepackage[bbgreekl]{mathbbol}
\usepackage{graphicx, color}
\usepackage{tikz}
\usetikzlibrary{cd, intersections, calc, decorations.pathmorphing, arrows, decorations.pathreplacing}
\usepackage[all]{xy}

\DeclareSymbolFontAlphabet{\mathbb}{AMSb}
\DeclareSymbolFontAlphabet{\mathbbl}{bbold}

\usepackage{stackengine}
\usepackage{calc}
\newlength\shlength
\newcommand\xshlongvec[2][0]{\setlength\shlength{#1pt}%
  \stackengine{-5pt}{$#2$}{\smash{$\kern\shlength%
    \stackengine{7.1pt}{$\mathchar"017E$}%
      {\rule{\widthof{$#2$}}{.57pt}\kern.4pt}{O}{r}{F}{F}{L}\kern-\shlength$}}%
      {O}{c}{F}{T}{S}}

\definecolor{refkey}{rgb}{0.9451,0.2706,0.4941}
\definecolor{labelkey}{rgb}{0.9451,0.2706,0.4941}

\usepackage{subfiles}
\usepackage[colorlinks, linkcolor=blue, citecolor=red, urlcolor=cyan,pagebackref]{hyperref} 

\numberwithin{equation}{section}

\usepackage{cleveref}
\crefname{thm}{Theorem}{Theorems}
\crefname{cor}{Corollary}{Corollaries}
\crefname{lem}{Lemma}{Lemmas}
\crefname{sublem}{Sublemma}{Sublemmas}
\crefname{prop}{Proposition}{Propositions}
\crefname{dfn}{Definition}{Definitions}
\crefname{defi}{Definition}{Definitions}
\crefname{ex}{Example}{Examples}
\crefname{claim}{Claim}{Claims}
\crefname{conj}{Conjecture}{Conjectures}
\crefname{conv}{Convention}{Conventions}
\crefname{rem}{Remark}{Remarks}
\crefname{rmk}{Remark}{Remarks}
\crefname{figure}{Figure}{Figures}
\crefname{section}{Section}{Sections}
\crefname{table}{Table}{Tables}

\newtheorem{thm}{Theorem}[section]
\newtheorem{prop}[thm]{Proposition}
\newtheorem{cor}[thm]{Corollary}
\newtheorem{lem}[thm]{Lemma}

\theoremstyle{definition}
\newtheorem{dfn}[thm]{Definition}
\newtheorem{defi}[thm]{Definition}
\newtheorem{ex}[thm]{Example}

\theoremstyle{remark}
\newtheorem{rmk}[thm]{Remark}
\newtheorem{rem}[thm]{Remark}
\newtheorem{prob}[thm]{Problem}

\DeclareMathOperator{\precceq}{\text{$\preccurlyeq$}}
\DeclareMathOperator{\precov}{\text{$\prec\!\!\!\cdot$}}
\DeclareMathOperator{\succceq}{\text{$\succcurlyeq$}}
\DeclareMathOperator{\succov}{\text{$\cdot\!\!\!\succ$}}

\newcommand{\ot}{\leftarrow}

\newcommand*{\chom}{\mathcal{H}\kern -.5pt om}

\newcommand{\bZ}{\mathbb{Z}}
\newcommand{\bQ}{\mathbb{Q}}
\newcommand{\bR}{\mathbb{R}}
\newcommand{\bC}{\mathbb{C}}

\newcommand{\bS}{\mathbb{S}}

\newcommand{\bE}{\mathbb{E}}

\newcommand{\A}{\mathcal{A}}
\newcommand{\cA}{\mathcal{A}}

\newcommand{\cC}{\mathcal{C}}

\newcommand{\cF}{\mathcal{F}}

\newcommand{\X}{\mathcal{X}}
\newcommand{\cX}{\mathcal{X}}

\newcommand{\sgn}{\mathrm{sgn}}
\newcommand{\trop}{\mathrm{trop}}
\newcommand{\stab}{\mathrm{stab}}

\newcommand{\tr}{\mathsf{T}}

\DeclareMathOperator{\interior}{\mathrm{int}}
\DeclareMathOperator{\bdim}{\mathbf{dim}}

\newcommand{\bs}{{\boldsymbol{s}}}

\newcommand{\ve}{b}
\newcommand{\bep}{\boldsymbol{\epsilon}}

\newcommand{\indi}{i}

\newcommand{\indk}{k}

\newcommand{\bx}{\mathbf{x}}
\newcommand{\bX}{\mathbf{X}}

\newcommand{\bExch}{\bE \mathrm{xch}}

\newcommand{\Teich}{Teichm\"uller}

\makeatletter
\newcommand{\oset}[3][0ex]{%
  \mathrel{\mathop{#3}\limits^{
    \vbox to#1{\kern-2\ex@
    \hbox{$\scriptstyle#2$}\vss}}}}
\makeatother
\newcommand{\overbar}[1]{\oset{#1}{-\!\!\!-\!\!\!-}}

\makeatletter
\newcommand{\osetnear}[3][0ex]{%
  \mathrel{\mathop{#3}\limits^{
    \vbox to#1{\kern-.3\ex@
    \hbox{$\scriptstyle#2$}\vss}}}}
\makeatother
\newcommand{\overbarnear}[1]{\osetnear{#1}{-\!\!\!-\!\!\!-}}

\newcommand\qarrowop[2]{\draw[->,>=latex,shorten >=2pt,shorten <=2pt] (#2) -- (#1) [thick];} 

\tikzset{
  mid arrow/.style={postaction={decorate,decoration={
        markings,
        mark=at position .5 with {\arrow[#1]{stealth}}
      }}},
}

\title[Entropy and the trichotomy of acyclic quivers]
{Entropy of cluster DT transformations and\\ the finite-tame-wild trichotomy of acyclic quivers}

\author{Tsukasa Ishibashi}
\address{Tsukasa Ishibashi, Mathematical Institute, Tohoku University, 
6-3 Aoba, Aramaki, Aoba-ku, Sendai, Miyagi 980-8578, Japan.}
\email{tsukasa.ishibashi.a6@tohoku.ac.jp}
\urladdr{https://sites.google.com/view/tsukasa-ishibashi/home} 

\author{Shunsuke Kano}
\address{Shunsuke Kano, Mathematical Science Center for Co-creative Society, Tohoku University, 
468-1 Aoba, Aramaki, Aoba-ku, Sendai, Miyagi 980-0845, Japan.}
\email{s.kano@tohoku.ac.jp}
\urladdr{https://sites.google.com/view/shunsuke-kano} 

\date{\today}

\begin{document}
\maketitle

\begin{abstract}
The cluster algebra associated with an acyclic quiver has a special mutation loop $\tau$, called the cluster Donaldson--Thomas (DT) transformation, related to the Auslander--Reiten translation.
In this paper, we characterize the finite-tame-wild trichotomy for acyclic quivers by the sign stability of $\tau$ introduced in \cite{IK19} and its cluster stretch factor. 
As an application, we compute several kinds of entropies of $\tau$ and other mutation loops. 
In particular, we show that the algebraic and categorical entropies of $\tau$ are commonly given by the 
logarithm of the spectral radius
of the Coxeter matrix associated with the quiver, and that any mutation loop of finite or tame acyclic quivers have zero algebraic entropy.
\end{abstract}

\tableofcontents


\section{Introduction}

From a quiver $Q$, one can construct several interesting algebraic/geometric objects. One such construction is the \emph{cluster algebra} formulated by Fomin--Zelevinsky \cite{FZ-CA1}, and the \emph{cluster variety} formulated by Fock--Goncharov \cite{FG09} as its geometric counterpart. 

The cluster algebra/variety associated with a quiver $Q$ has a natural automorphism group called the \emph{cluster modular group} \cite{FG09}, whose elements are represented by  
sequences of mutations 
\begin{align}\label{eq:mutation_loop}
    Q = Q_0 \overbar{\mu_{k_0}} Q_1 \overbar{\mu_{k_1}} \cdots \overbar{\mu_{k_{n-1}}} Q_n, \quad Q_n = Q
\end{align}
that return back to the original quiver. 

\subsection{Sign stability}
When the quiver $Q$ is associated with an ideal triangulation of a marked surface $\Sigma$, the corresponding cluster theory is beautifully connected to the \Teich\ theory \cite{Penner,FST}, where the cluster modular group coincides with the mapping class group of $\Sigma$ (up to finite index) \cite{BS}.  
In the previous work \cite{IK19}, the authors introduced a new property of mutation sequences \eqref{eq:mutation_loop} called the \emph{sign stability}, as a cluster-theoretic analogy of the pseudo-Anosov mapping classes of a surface. The sign stability is defined by resembling some combinatorial property of pseudo-Anosov mapping classes, and it has been shown that the sign stability well reproduces some of the important dynamical properties of the pseudo-Anosov mapping classes \cite{IK19,IK20a} 

We have several variants of the notion of sign stability, each of which has its own advantage. 
The \emph{basic} sign stability is a kind of the weakest version which is sufficient to compute
\begin{itemize}
    \item the algebraic entropy of the corresponding cluster transformations \cite{IK19}, and 
    \item the categorical entropy of some corresponding derived equivalence of the Ginzburg dg algebra of the quiver (with potential) \cite{Kan21}.
\end{itemize}
A mutation sequence \eqref{eq:mutation_loop} with the basic sign stability has the numerical invariant called the \emph{cluster stretch factor}.
Then the entropies above are commonly given by the logarithm of the cluster stretch factor.
We remark that the basic sign stability also corresponds to a weaker version of the \emph{asymptotic sign coherence} \cite{GN}. (To be written in \cite{IK}.)

\subsection{Finite-tame-wild trichotomy for acyclic quivers}


A quiver $Q$ is said to be \emph{acyclic} if it has no oriented cycles. In this case, one can construct its \emph{path algebra}, and its representation theory has been studied by many authors. 
In particular, every acyclic quiver is classified into the following three types depending on the representation theory of its path algebra: \emph{representation finite}, \emph{tame}, or \emph{wild}.
The latter two types are collectively called \emph{representation infinite}.
The trichotomy -- representation finite, tame or wild -- is characterized in several ways:
\begin{itemize}
    \item The underlying graph $\Delta$ of $Q$ is Dynkin, affine (\emph{a.k.a.} Euclidean) or of other types \cite{Gab,DF,Naz}.
    \item The Cartan matrix $A$ of $Q$ is positive definite, positive semi-definite, or negative definite \cite{Kac}.
    \item The supremum of the multiplicities of the arrows occurring in all the quivers mutation-equivalent to $Q$ is 1, 2, or at least 3 \cite{Kel_tri}.
    \item The frieze variety of $Q$ is of dimension 0, 1 or at least 2 \cite{LLMSS}.
\end{itemize}

\subsection{Statements}
In this paper, we give a new characterization of the finite-tame-wild trichotomy for acyclic quivers in terms of the sign stability.

It is known that every acyclic quiver admits a special mutation sequence \eqref{eq:mutation_loop} called a \emph{reddening sequence}. Explicitly, it is given by the mutation sequence $\gamma_{\pi}$ associated with any \emph{admissible labeling} $\pi$ of the vertices of $Q$. 
It represents the canonical central element in the cluster modular group called the \emph{cluster Donaldson--Thomas (DT) transformation}. 
The cluster DT transformation gives the Kontsevich--Soibelman's noncommutative DT invariant \cite{KS} associated with the quiver \cite{Kel_der}, and also corresponds to the Auslander--Reiten translation on the derived category of the path algebra \cite{ASS12}. 

Here is our main statement:
\begin{thm}[{\cref{thm:trich}}]
Let $Q$ be an acyclic quiver, and $\gamma_\pi$ be the mutation sequence associated with any admissible labeling $\pi$ of $Q$.
Then,
\begin{enumerate}
    \item $Q$ is representation finite if and only if $\gamma_\pi$ is not basic sign-stable.
    \item $Q$ is tame if and only if $\gamma_\pi$ is basic sign-stable with cluster stretch factor 1.
    \item $Q$ is wild if and only if $\gamma_\pi$ is basic sign-stable with cluster stretch factor larger than 1.
\end{enumerate}
\end{thm}

In particular, we can compute the entropies of the cluster DT transformation $\tau$
and its inverse $\tau^{-1}$ as we mentioned above:

\begin{thm}[\cref{thm:ent_tau}]
Let $Q$ be a representation infinite acyclic quiver.
Then, the following quantities coincide with each other:
\begin{itemize}
    \item the logarithms of the cluster stretch factors $\lambda_\tau$, $\lambda_{\tau^{-1}}$;
    \item the algebraic entropies $h_\mathrm{alg}(\tau_a^{\pm 1})$, $h_\mathrm{alg}(\tau_x^{\pm 1})$ of the associated cluster $\A$- and $\X$-transformations, respectively;
    \item the categorical entropies $h_T(F_{\tau^{\pm 1}}|_{\mathsf{D}_\mathsf{fd}})$, $h_0(F_{\tau^{\pm 1}}|_{\mathsf{per}})$ of the associated derived autoequivalences;
    \item the spectral radius $\rho(\Phi)$ of the Coxeter matrix $\Phi$ of $Q$.
\end{itemize}
\end{thm}
In addition, we show that the algebraic entropies of any mutation loop of a representation finite or tame acyclic quiver are zero (\cref{thm:ent_vanish_tame}).

\smallskip
\paragraph{\textbf{A proposal of generalization}}
To prove our main results, we use some specific properties for acyclic quivers.
However, the statement itself can be generalized for any quivers having cluster DT transformations.
We conclude this introduction by proposing the following problem:

\begin{prob}
Formulate and prove a finite-tame-wild trichotomy for any quivers having cluster DT transformations via sign stability and their cluster stretch factor.
\end{prob}
We expect that existing works on the quivers of finite mutation type \cite{FST12,FSTT14,Ish20,GK21} will be effective in characterizing/defining the `tame' class.




\subsection*{Organization of the paper}
In \cref{sec:SS}, we recall basic notions in the theory of cluster algebras and the definition of sign stability.
In \cref{sec:DT_Coxeter}, we prove the sign stability of the cluster Donaldson--Thomas transformations of representation infinite acyclic quivers.
The main theorem is proved in \cref{sec:trich}.
In \cref{sec:dyn}, we compute several kinds of entropies and study the dynamical property of the cluster Donaldson--Thomas transformation.
Some lemmas are proved in \cref{sec:proofs} by using the representation theory of quivers.

\subsection*{Acknowledgements}
The authors are grateful to Atsushi Takahashi for giving us the definitive comment on the proof of \cref{lem:M^tr*Phi,lem:M*Phi^-1}.
T. I. is supported by JSPS KAKENHI (20K22304).
S. K. is partially supported by scientific research support of the Research Alliance Center for Mathematical Sciences and Mathematical Science Center for Co-creative Society, Tohoku University.
\section{Sign stability}
\label{sec:SS}

In this section, we recall the basic terminologies around cluster algebras and the notion of sign stability.

\subsection{Seeds, mutations and the labeled exchange graph}\label{subsec:seeds}

Fix a finite set $I=\{0,\dots,N-1\}$ of indices and a field $\cF_X$ isomorphic to the field $\bQ(z_0,\dots,z_{N-1})$ of rational functions on $N$ variables. 
A \emph{(labeled) seed} in $\cF_X$ is a triple $(Q,\bx)$, where
\begin{itemize}
    \item $Q$ is a quiver having no loops nor 2-cycles, whose vertices being parametrized by the set $I$.
    \item $\bX=(X_i)_{i \in I}$ is a tuple of algebraically independent elements (called the \emph{cluster $\X$-variables}) in $\cF_X$.
\end{itemize}
The data of quiver $Q$ is encoded in the skew-symmetric matrix $B = (b_{ij})_{i,j\in I}$ (called the \emph{exchange matrix}), defined by 
\begin{align*}
    b_{ij} := \#\{\text{the arrows from $i$ to $j$ in $Q$}\} - \#\{\text{the arrows from $j$ to $i$ in $Q$}\}.
\end{align*}
We will identify them when no confusion can occur.

For an index $k \in I$, the \emph{mutation} directed to $k$ produces a new seed $(Q', \bX') = \mu_k(Q,\bX)$ by the formula
\begin{align}
    \ve'_{ij} &:= 
    \begin{cases}
    -\ve_{ij} & \mbox{if $i=k$ or $j=k$}, \\
    \ve_{ij} + [\ve_{ik}]_+ [\ve_{kj}]_+ - [-\ve_{ik}]_+ [-\ve_{kj}]_+ & \mbox{otherwise},
    \end{cases} \label{eq:matrix mutation}\\
    X'_i&:= 
    \begin{cases}
    X_k^{-1} & i=k,\\
    X_i\,(1 + X_k^{-\mathrm{sgn}(\ve_{ik})})^{-\ve_{ik}} & i \neq k.
  \end{cases} \label{eq:X-transf}
\end{align}
A permutation $\sigma \in \mathfrak{S}_{I}$ produces a new seed $(Q',\bX')=\sigma(Q,\bX)$ by the rule
\begin{align}\label{eq:seed_permutation}
    \ve'_{ij}:=\ve_{\sigma^{-1}(i),\sigma^{-1}(j)}, \quad 
    X'_i:=X_{\sigma^{-1}(i)}.
\end{align}
We say that two seeds in $\cF_X$ are \emph{mutation-equivalent} if they are transformed to each other by a finite sequence of mutations and permutations. The equivalence class is usually called a \emph{mutation class}. 

\begin{dfn}
The relations among the seeds in a given mutation class $\bs$ can be encoded in the \emph{(labeled) exchange graph} $\bExch_\bs$. It is a graph with vertices $v$ corresponding to the seeds $\bs^{(v)}$ in $\bs$, together with labeled edges of the following two types:
\begin{itemize}
    \item (horizontal edge) edges of the form $v \overbar{k} v'$ whenever the seeds $\bs^{(v)}$ and $\bs^{(v')}$ are related by the mutation $\mu_k$ for $k \in I$;
    \item (vertical edge) edges of the form $v \overbarnear{\sigma} v'$ whenever the seeds $\bs^{(v)}$ and $\bs^{(v')}$ are related by the transposition $\sigma=(j\ k)$ for $(j,k) \in I \times I$.
\end{itemize}
An edge path in $\bExch_\bs$ corresponds to the usual notion of a \emph{mutation sequence}. 
When no confusion can occur, we simply denote a vertex of the labeled exchange graph by $v \in \bExch_\bs$ instead of $v \in V(\bExch_\bs)$. 
\end{dfn}

We will write $\bs^{(v)} = (Q^{(v)}, \bX^{(v)})$ for $v \in \bExch_\bs$.

\begin{rem}[mutation class from a quiver] \label{ld:attachment_generation}
Given a quiver $Q$ labeled by $I$, consider the seed $(Q, (X_i)_{i \in I})$ in $\cF_X:=\bQ(X_i \mid i \in I)$ and its mutation class $\bs_Q$.
Then the labeled exchange graph $\bExch_{\bs_Q}$ depends only on the mutation class of $Q$. Indeed, it is unchanged if we transform the cluster variables simultaneously by an automorphism of the ambient field. We put $Q=Q^{(v_0)}$ at a vertex $v_0$ of this graph, and call it the \emph{initial quiver} of the mutation class $\bs_Q$. 
\end{rem}

\begin{ex}[Type $A_2$]\label{ex:A2_exch}
Let us consider the quiver $Q = (0 \to 1)$ of type $A_2$.
Namely, let $I := \{0, 1\}$ and consider the exchange matrix $B = \begin{psmallmatrix} 0 & 1 \\ -1 & 0 \end{psmallmatrix}$.
Then, the labeled exchange graph $\bExch_\bs$ of the mutation class $\bs = \bs_Q$ with initial seed $\bs^{(v_0)} = (Q, (X_0, X_1))$ is a finite graph as shown in \cref{f:A_2 exch}.
In this figure, red (resp. blue) vertices correspond to the seeds with the underlying quiver $Q$ (resp. $Q' = (0 \ot 1)$).

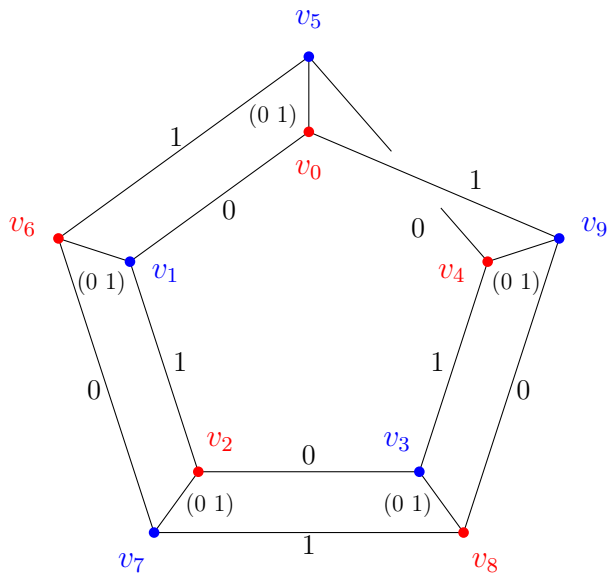
\begin{figure}[h]
    \centering
\begin{tikzpicture}
\foreach \i in {0,1,2,3,4} 
{
    \draw (72*\i+90:2.5) coordinate(A\i);
    \draw (72*\i+90:3.5) coordinate(B\i);
}
\foreach \i in {0,1,2,3} 
{
    \draw (72*\i+90:2.5) -- (72*\i+90+72:2.5);
    \draw (72*\i+90:3.5) -- (72*\i+90+72:3.5);
}
\draw (A4) -- (B0);
\fill[white] ($(A4)!0.4!(B0)$) circle(0.5cm);
\draw (B4) -- (A0);
\foreach \i in {0,1,2,3,4} \draw (A\i) -- (B\i);
\foreach \i in {0,2,4} 
{
    \fill[red] (A\i) circle(2pt);
    \fill[blue] (B\i) circle(2pt);
    \draw (A\i)++(72*\i+90:-0.5) node[red]{$v_\i$};
    \pgfmathsetmacro{\j}{int(\i+5)}
    \draw (B\i)++(72*\i+90:0.5) node[blue]{$v_\j$};
}
\foreach \i in {1,3} 
{
    \fill[blue] (A\i) circle(2pt);
    \fill[red] (B\i) circle(2pt);
    \draw (A\i)++(72*\i+90:-0.5) node[blue]{$v_\i$};
    \pgfmathsetmacro{\j}{int(\i+5)}
    \draw (B\i)++(72*\i+90:0.5) node[red]{$v_\j$};
}
\foreach \i in {1,3}
{
    \draw(72*\i+90-36:1.8) node[scale=0.9]{$0$};
    \draw(72*\i+90-36:3) node[scale=0.9]{$1$};
}
\foreach \i in {2,4}
{
    \draw(72*\i+90-36:1.8) node[scale=0.9]{$1$};
    \draw(72*\i+90-36:3) node[scale=0.9]{$0$};
}
\draw(90-50:1.9) node[scale=0.9]{$0$};
\draw(90-50:2.9) node[scale=0.9]{$1$};
\draw(0+100:2.75) node[scale=0.75]{$(0\ 1)$};
\draw(72+98:2.8) node[scale=0.75]{$(0\ 1)$};
\draw(-72+82:2.8) node[scale=0.75]{$(0\ 1)$};
\draw(144+98:2.8) node[scale=0.75]{$(0\ 1)$};
\draw(-144+82:2.8) node[scale=0.75]{$(0\ 1)$};
\end{tikzpicture}
    \caption{The labeled exchange graph of type $A_2$.}
    \label{f:A_2 exch}
\end{figure}
\end{ex}

\subsection{Tropical cluster variety}

Usually, we first introduce the notion of cluster variety (a positive scheme), and then the tropical cluster variety is defined as the set of its $\bR^\trop$-valued points, where $\bR^\trop=(\bR,\min,+)$ denotes the tropical semifield.
However, we do not use the cluster variety itself in the main part of this paper, so we directly introduce the tropical cluster variety.
The following definition is equal to the above mentioned one.

Let $f(X_0, \dots, X_{N-1})$ be a positive rational function on $N$ variables (namely, a rational function admitting a subtraction-free expression). 
Then its tropical limit $f^\trop(x_0, \dots, x_{N-1})$ is defined by
\begin{align}\label{eq:trop_limit}
    f^\trop(x_0, \dots, x_{N-1}) 
    := \lim_{\epsilon \to -0} \epsilon \log f(e^{x_0/\epsilon}, \dots, e^{x_{N-1}/\epsilon}),
\end{align}
which is a piecewise linear function on $\bR^I$. 
Observe that the cluster $\X$-transformation formula \eqref{eq:X-transf} is an $N$-component positive rational function. Its tropical limit $\bx \mapsto \bx'$, which is explicitly given by
\begin{align}\label{eq:X_trop_mut}
    x'_i =
    \begin{cases}
        -x_k & \mbox{if } i = k,\\
        x_i - \ve_{ik} \min\{0, -\sgn(\ve_{ik}) x_k\} & \mbox{if } i \neq k,
    \end{cases}
\end{align}
is called the \emph{tropical cluster $\X$-transformation}.

\begin{defi}\label{def:X^trop}
The \emph{tropical $\X$-variety} $\X_\bs(\bR^\trop)$ associated with a mutation class $\bs$ is a piecewise linear manifold homeomorphic to $\bR^I$, equipped with the distinguished atlas consisting of global charts $\bx^{(v)}: \cX_\bs(\bR^\trop) \xrightarrow{\sim} \bR^I$ parametrized by the vertices $v \in \bExch_\bs$ such that the coordinate transformations among them are given by the tropical cluster $\cX$-transformations and permutations.
\end{defi}

We note that there is an $\bR_{>0}$-action on $\cX_\bs(\bR^\trop)$ so that 
\begin{align*}
    \bx^{(v)}(t \cdot w) = t \bx^{(v)}(w)
\end{align*}
for any $t \in \bR_{>0}$, $w \in \cX_\bs(\bR^\trop)$ and $v \in \bExch_\bs$.

\subsection{Cluster modular group}\label{subsec:cluster_modular}
Given a mutation class $\bs$ of seeds, let $\mathrm{Quiv}_\bs$ denote the mutation class of quivers underlying $\bs$. Then we have a map 
\begin{align*}
    Q^\bullet: V(\bExch_\bs) \to \mathrm{Quiv}_\bs, \quad v \mapsto Q^{(v)}.
\end{align*}

\begin{dfn}
The \emph{cluster modular group} $\Gamma_\bs \subset \mathrm{Aut}(\bExch_\bs)$ consists of graph automorphisms $\phi$ which preserve the fibers of the map $Q^\bullet$ and the labels on the edges (in particular, the horizontal/vertical properties). An element of the cluster modular group is called a \emph{mutation loop}. 
\end{dfn}

The cluster modular group acts on the tropical cluster $\cX$-variety $\cX_\bs(\bR^\trop)$ piecewise-linearly 
so that
\begin{align*}
    \bx^{(v)}_i(\phi(w)) = \bx^{(\phi^{-1}(v))}_i(w)
\end{align*}
for any $\phi \in \Gamma_\bs$, $w \in \X_\bs(\bR^\trop)$, $v \in \bExch_\bs$ and $i \in I$. 

For an explicit computation of the action in terms of coordinates, the following description is useful. 
Given $\phi \in \Gamma_\bs$ and a vertex $v_0 \in \bExch_\bs$, there exists an edge path $\gamma$ from $v_0$ to $v:=\phi^{-1}(v_0)$.
We call $\gamma$ a \emph{representation path} of $\phi$. 
Associated to $\gamma$ is a sequence $\mu_\gamma$ of mutations and permutations, satisfying the condition $Q^{(v)}=Q^{(v_0)}$. 
Then, the action $\phi: \cX_\bs(\bR^\trop) \to \cX_\bs(\bR^\trop)$ fits into the following diagram:
\begin{equation*}
\begin{tikzcd}
    \cX_\bs(\bR^\trop) \ar[r, equal] \ar[d, "\bx^{(v_0)}"] & \cX_\bs(\bR^\trop) \ar[r, "\phi"] \ar[d, "\bx^{(v)}"] & \cX_\bs(\bR^\trop) \ar[d, "\bx^{(v_0)}"]\\
    \bR^I \ar[r, "\mu_\gamma"] & \bR^I \ar[r, "\sim"] & \bR^I
\end{tikzcd}
\end{equation*}
Here the right bottom isomorphism is given by $x_i^{(v_0)} \mapsto x_i^{(v)}$ for all $i \in I$.

We say that an edge path $\gamma$ in $\bExch_\bs$ is \emph{horizontal} if it only traverses horizontal edges. The examples of mutation loops on which we focus in this paper always admit horizontal representation paths. For the sake of simplicity, we will recall the sign stability only for such mutation loops in the next section.  

\begin{ex}[Type $A_2$]\label{ex:A2_mut_loop}
We continue to consider the mutation class $\bs$ in \cref{ex:A2_exch}.
One can verify that the cluster modular group $\Gamma_\bs$ is isomorphic to $\bZ/5\bZ$, and its generator $\phi$ is given by the composition of the $2\pi/5$-rotation $v_i \mapsto v_{i+1}$ and the involution such that $v_i \mapsto v_{i+5}$. It has a representation path $v_0 \overbar{0} v_1 \overbar{(0\ 1)} v_6$. 
Note that $\phi^2$ is the $4\pi/5$-rotation  $v_i \mapsto v_{i+2}$, which admits a horizontal representation path $v_0 \overbar{0} v_1 \overbar{1} v_2$.
\end{ex}

\begin{rem}
Given a mutation loop $\phi \in \Gamma_\bs$, the coordinate expression $\phi_{(v_0)}$ depends only on the initial vertex $v_0$ but not on the representation path $\gamma$. On the other hand, $\gamma$ gives its factorization to a composite of cluster transformations. The notion of \emph{sign stability} will be defined for the latter data. 
\end{rem}

\subsection{Sign stability}\label{subsec:sign stability}

Now recall the tropical cluster $\X$-transformation $\mu_k = \bx^{(v')} \circ (\bx^{(v)})^{-1}: \bR^I \to \bR^I$ associated with an edge $v \overbar{k} v'$ in $\bExch_\bs$.
For a real number $a \in \bR$, let $\sgn(a)$ denote its sign:
\[
\sgn(a):=
\begin{cases}
    + & \mbox{ if } a>0,\\
    0 & \mbox{ if } a=0,\\
    - & \mbox{ if } a<0.
\end{cases}
\]
We say that $\sgn(a)$ is \emph{strict} if $\sgn(a)\neq 0$. 
The following expression is useful in the sequel:

\begin{lem}[{\cite[Lemma 3.1]{IK19}}]\label{l:x-cluster signed}
The tropical cluster $\X$-transformation of its coordinates can be written as
\begin{align}\label{eq:sign x-cluster}
    x^{(v')}_\indi =
\begin{cases}
    -x^{(v)}_\indk & \mbox{if $\indi=\indk$}, \\
    x^{(v)}_\indi+[\sgn(x^{(v)}_\indk) b^{(v)}_{\indi\indk}]_+x^{(v)}_\indk & \mbox{if $\indi \neq \indk$}.
\end{cases}
\end{align}
\end{lem}



We are going to define the \emph{sign} of an edge path in $\bExch_\bs$.
In what follows, we only consider horizontal edge paths (\emph{i.e.}, involving no permutations) for simplicity.




\begin{dfn}[sign of a path]\label{d:sign}
Given a horizontal path $\gamma:  v_0 \overbar{k_0} v_1 \overbar{k_1} \cdots \overbar{k_{h-1}} v_h$ in $\bExch_\bs$ and a point $w \in \X_\bs(\bR^\trop)$, the \emph{sign} $\boldsymbol{\epsilon}_\gamma(w)$ of $\gamma$ at $w$ is the sequence
\begin{align*}
    \boldsymbol{\epsilon}_\gamma(w)=(\sgn(x^{(v_0)}_{k_0}(w)), \dots, \sgn(x^{(v_{h-1})}_{k_{h-1}}(w))).
\end{align*}
\end{dfn}

The sign $\bep_\gamma(w)$ tells us the domain of linearity of the PL isomorphism $\mu_\gamma: \bR^I \to \bR^I$.
For $\bep \in \{+, -\}^h$, we define
\begin{align*}
    \cC^{\bep}_\gamma := \overline{\{ w \in \cX_\bs(\bR^\trop) \mid \bep_\gamma(w) = \bep \}}.
\end{align*}
This subspace is a cone in $\cX_\bs(\bR^\trop)$.

\begin{lem}[{\cite[Lemma 3.5]{IK19}}]
If $\cC^{\bep}_\gamma \neq \emptyset$, then $\bx^{(v_0)}(\cC^{\bep}_\gamma)$ is a domain of linearity of the piecewise linear isomorphism $\mu_\gamma: \bR^I \to \bR^I$.
\end{lem}

We denote by $E^{\bep}_\gamma$ the presentation matrix of the linear extension of the restriction $\mu_\gamma|_{\bx^{(v_0)}(\cC^{\bep}_\gamma)}$.

\begin{ex}
Let 
\begin{align}\label{eq:C^+-}
    \cC^\pm_{(v)} := \{ w \in \cX_\bs(\bR^\trop) \mid \pm x_i^{(v)}(w) \geq 0 \text{ for all } i \in I \}
\end{align}
for each $v \in \bExch_\bs$.
Then, the sign $\bep_\gamma(w)$ at $w \in \interior \cC^+_{(v_0)}$ of any edge path $\gamma$ starting from $v_0$ coincides with the sign of \emph{$c$-vectors}, called the \emph{tropical sign}, along $\gamma$.
Therefore, the presentation matrix on $\cC^+_{(v_0)}$ coincides with the \emph{$C$-matrix} assigned at the terminal vertex of $\gamma$.
We refer the reader to \cite{IK19} for the details.
\end{ex}

\begin{dfn}[sign stability]\label{d:sign stability}
Let $\phi \in \Gamma_\bs$ be a mutation loop, and let $\gamma$ be its horizontal representation path starting from $v_0 \in \bExch_\bs$.
Let $\Omega \subset \X_\bs(\bR^\trop)$ be a subset which is invariant under the rescaling action of $\bR_{> 0}$, which we call a \emph{domain of stability}. 

Then we say that $\gamma$ is \emph{sign-stable} on $\Omega$ if there exists a sequence $\bep^\stab \in \{+,-\}^h$ of strict signs such that for each $w \in \Omega \setminus \{0\}$, there exists an integer $n_0 \in \mathbb{N}$ such that   \[\boldsymbol{\epsilon}_\gamma(\phi^n(w)) = \boldsymbol{\epsilon}^\stab \]
for all $n \geq n_0$. 
We call $\bep^\stab=\boldsymbol{\epsilon}_{\gamma,\Omega}^\stab$ the \emph{stable sign} of $\gamma$ on $\Omega$.
Also, we write $E^{(v_0)}_{\phi, \Omega} := E^{\bep^\stab}_\gamma$ for the presentation matrix and call it \emph{stable presentation matrix}.
\end{dfn}

We note that the stable presentation matrix $E^{(v_0)}_{\phi, \Omega}$ depends only on the mutation loop $\phi$ and the initial vertex $v_0$, since it is a presentation matrix of the restriction of the action $\phi: \cX_\bs(\bR^\trop) \to \cX_\bs(\bR^\trop)$ to one of its domain of linearity with respect to the coordinate system $\bx^{(v_0)}$ (\emph{cf}. \cite[Corollary 3.7]{IK19}).

We are going to mention the Perron--Frobenius property of a sign-stable mutation loop. 
We say that 
a horizontal edge path $\gamma: v_0 \overbar{k_0} v_1 \overbar{k_1} \cdots \overbar{k_{h-1}} v_h$ in $\bExch_\bs$ is \emph{fully-mutating} if
\begin{align*}
    \{k_0, k_1, \dots, k_{h-1} \} = I.
\end{align*}
An $\bR_{\geq 0}$-invariant set $\Omega \subset \cX_\bs(\bR^\trop)$ is said to be \emph{tame} if it has a non-empty intersection with the set $\bigcup_{v \in \bExch_\bs} \interior \cC^+_{(v)}$.

\begin{thm}[Perron--Frobenius property, {\cite[Theorem 3.12]{IK19}}]\label{thm:Perron-Frobenius}
Suppose $\gamma$ is a fully-mutating edge path which represents a mutation loop $\phi$, and sign-stable on a tame subset $\Omega$. 
Then the spectral radius of the stable presentation matrix $E^{(v_0)}_{\phi,\Omega}$ is attained by a positive eigenvalue $\lambda_{\phi,\Omega} \geq 1$.
\end{thm}

Every path $\gamma$ starting from $v_0$ has a constant sign in the interior of the cone $\cC^+_{(v_0)}$ (resp. $\cC^-_{(v_0)}$), which is given by the tropical sign (resp. that for the opposite mutation class) \cite[Lemma 3.12, Corollary 3.13]{IK19}. 
In this sense the sign stability on the set
\begin{align}\label{eq:Omega^can}
    \Omega^{\mathrm{can}}_{(v_0)}:=\interior\cC^+_{(v_0)} \cup \interior\cC^-_{(v_0)}
\end{align}
is most fundamental, and it turns out that it is sufficient for the computation of the algebraic entropy of cluster transformations (\cref{subsec:entropy}).
\begin{dfn}\label{d:cluster stretch factor}
A representation path $\gamma$ starting from $v_0 \in \bExch_\bs$ of a mutation loop $\phi \in \Gamma_\bs$ is \emph{basic sign-stable} if it is sign-stable on $\Omega^\mathrm{can}_{(v_0)}$.
In this case, we call
$\lambda_{\phi}^{(v_0)}:=\lambda_{\phi,\Omega^{\mathrm{can}}_{(v_0)}}$ the \emph{cluster stretch factor} of $\phi$.
\end{dfn}

We remark that the cluster stretch factor $\lambda_\phi^{(v_0)}$ does not depend on $v_0$ if it is fully-mutating or $\lambda_\phi^{(v_0)} >1$ \cite[Remark 3.16]{IK19}.
\section{Cluster Donaldson--Thomas transformation and Coxeter matrix}\label{sec:DT_Coxeter}

Let $Q$ be an acyclic quiver with the set of vertices $Q_0$, and $I:=\{0,\dots,N-1\}$ with $N:=|Q_0|$.
Choose a bijective labeling $\pi: I \to Q_0$ of vertices by $I$ so that $i < j$ if there is an arrow $\pi(i) \to \pi(j)$.
We call such a bijection $\pi$ an \emph{admissible labeling}, following \cite{ASS}.
An example is shown in \cref{fig:acyclic_example}. 
For $i \in I$, we abbreviate the vertex $\pi(i) \in Q_0$ as $i$.
For $i,j \in I$, let
\begin{align*}
    a_{ij}:= \begin{cases}
    n & \mbox{if there is $n$ arrows $i \to j$}, \\
    0 & \mbox{otherwise}.
    \end{cases}
\end{align*}
Then $a_{ij} \geq 0$ for all $i,j \in I$, and $a_{ij}> 0$ only if $i < j$. Using the exchange matrix $(\ve_{ij})_{i,j \in I}$ of $Q$, we can write $a_{ij}=[\ve_{ij}]_+$. 

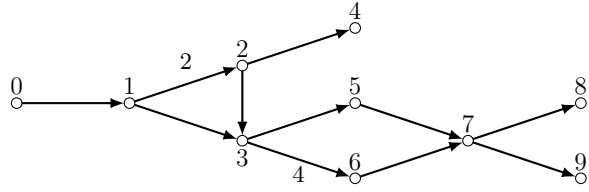
\begin{figure}[ht]
    \centering
\begin{tikzpicture}
\draw(0,0) circle(2pt) node[above,scale=0.8]{$0$};
\draw(1.5,0) circle(2pt) node[above,scale=0.8]{$1$};
\draw(3,0.5) circle(2pt) node[above,scale=0.8]{$2$};
\draw(3,-0.5) circle(2pt) node[below,scale=0.8]{$3$};
\draw(4.5,0) circle(2pt) node[above,scale=0.8]{$5$};
\draw(4.5,1) circle(2pt) node[above,scale=0.8]{$4$};
\draw(4.5,-1) circle(2pt) node[above,scale=0.8]{$6$};
\draw(6,-0.5) circle(2pt) node[above,scale=0.8]{$7$};
\draw(7.5,0) circle(2pt) node[above,scale=0.8]{$8$};
\draw(7.5,-1) circle(2pt) node[above,scale=0.8]{$9$};

\qarrowop{1.5,0}{0,0};
\qarrowop{3,0.5}{1.5,0};
\qarrowop{3,-0.5}{1.5,0};
\qarrowop{3,-0.5}{3,0.5};
\qarrowop{4.5,1}{3,0.5};
\qarrowop{4.5,0}{3,-0.5};
\qarrowop{4.5,-1}{3,-0.5};
\qarrowop{6,-0.5}{4.5,0};
\qarrowop{6,-0.5}{4.5,-1};
\qarrowop{7.5,0}{6,-0.5};
\qarrowop{7.5,-1}{6,-0.5};

\node at (2.25,0.55) {\scriptsize $2$};
\node at (3.75,-0.95) {\scriptsize $4$};
\end{tikzpicture}
    \caption{An example of acyclic quiver and its admissible labeling. The number $n$ of arrows between a pair of vertices is shown near the arrow only if $n>1$.}
    \label{fig:acyclic_example}
\end{figure}

We fix a mutation class $\bs = \bs_Q$ with the initial quiver $Q^{(v_0)} = Q$. See \cref{ld:attachment_generation}.

\begin{dfn}\label{def:incidence}
For $i, j \in I$, we use the following notation/terminology:
\begin{enumerate}
\item Write $i \precceq j$ if there exists a directed path from $i$ to $j$ in $Q$.
We regard there is a unique directed path from each vertex $i$ to itself.
Then $(I,\precceq)$ is a partially ordered set.
We write $i \prec j$ if $i \precceq j$ and $i \neq j$. 
\item Write $i \precov j$ if there exists an arrow $i \to j$ in $Q$. In particular, $i \precov j$ implies $i \prec j$. 
\item For each $\ast \in \{\precceq,\precov,\succceq,\succov\}$, define subsets $I_{\ast i}:=\{k \in I \mid k \ast i\}$.
\end{enumerate}
\end{dfn}

\subsection{Cluster Donaldson--Thomas transformation}

Consider the horizontal edge path
\begin{align}
    \gamma_\pi: v_0 \overbar{0} v_1 \overbar{1} v_2 \overbar{2}\cdots \overbar{N-1} v_{N}
\end{align}
in $\bExch_\bs$ determined by a choice of admissible labeling $\pi$.
The corresponding mutation sequence is called an \emph{admissible sequence}.
We write $Q^{[t]} := Q^{(v_t)}$ and $\bx^{[t]} := \bx^{(v_t)}$ for $t = 0, 1, \dots, N$.


\begin{rem}
We always use the notation/terminology in \cref{def:incidence} with respect to the initial quiver $Q=Q^{[0]}$. For example, we still say $0 \precov 1$ even if there is no arrow $0 \to 1$ in the mutated quiver $Q^{[1]}$. 
\end{rem}

The following statement is well-known:

\begin{lem}\label{lem:tau_is_mut_loop}
We have $Q^{[N]}=Q^{[0]}$. In particular, $\gamma_\pi$ represents a mutation loop $\tau \in \Gamma_\bs$. 
\end{lem}

\begin{proof}
This is well-known but we give a proof here for later use.  We claim that the vertex $i$ is a source in the quiver $Q^{[i]}$ for $i=0,\dots,N$. The case $i=0$ is clear from our choice of the labeling of vertices. 

Assume that the claim is true up to $i-1$. In the quiver $Q^{[t_-]}$ with $t_-:=\min I_{\precov i}$, the vertex $i$ is still not affected by mutations so that there are arrows $j \to i$ for each $j \in I_{\precov i}$, and arrows $i \to k$ for each $k \in I_{\succov i}$. Then we mutate all the vertices in $I_{\precov i}$ from $Q^{[t_-]}$, which are sources at that time by assumption, and the arrows of the former type are changed into $i \to j$, while any arrows of the latter type are unchanged.
Hence $i$ is a source in $Q^{[i]}$. The claim is proved. 

Therefore, each mutation step does not generate oriented cycles, and any arrows of $Q$ are two times reversed by applying the sequence of mutations.
This means that the desired statement holds.
\end{proof}
It will turn out that the mutation loop $\tau$ does not depend on the choice of admissible labeling (\cref{cor:clusterDT} below).

We are going to compute the PL action of $\tau$ on $\X_\bs(\bR^\trop)$. 
For $w \in \X_\bs(\bR^\trop)$, let us write $x_i(w):=x_i^{[0]}(w)$ for $i\in I$.
Then we define PL functions $\xi_i(w)$ on $\cX_\bs(\bR^\trop)$, $i \in I$ by $\xi_s(w):=x_s(w)$ for minimal vertices $s$ in $(I, \precceq)$ and the recursive relation
\begin{align}\label{eq:xi_definition}
    \xi_i(w) = x_i(w) + \sum_{j \precov i} a_{ji}[\xi_{j}(w)]_-
\end{align}
along $(I,\precceq)$. 
Here, $[a]_- := \min\{ 0, a \}$ for $a \in \bR$.

\begin{lem}\label{lem:xi}
We have $x_i^{[i]}(w)=\xi_i(w)$ for all $i \in I$ and $w \in \cX_\bs(\bR^\trop)$. 
\end{lem}

\begin{proof}



We proceed by induction on $i=0,\dots,N$.  
For $i=0$, it is true from the definition of $\xi_0(w)$. Assume that the assertion holds for all $t< i$. 

Under that assumption, we show that 
\begin{align}\label{eq:induction_2}
    x_i^{[t]}(w)=x_i(w)+ \sum_{s< t} a_{si}[\xi_s(w)]_-
\end{align}
holds for $t\leq i$. For $t=0$, $x_i^{[0]}(w)=x_i(w)$ is clear. Assume that \eqref{eq:induction_2} holds for some $t < i$. From the quiver $Q^{[t]}$, we are going to mutate at the vertex $t$. Here we have $x_t^{[t]}(w)=\xi_t(w)$ from the first induction assumption, and $x_i^{[t]}(w)=x_i(w)+ \sum_{s< t} a_{si}[\xi_s(w)]_-$ from the induction assumption for \eqref{eq:induction_2}. Then by the mutation formula \eqref{eq:X_trop_mut}, we get
\begin{align*}
    x_i^{[t+1]}(w) = x_i^{[t]}(w)+ a_{ti}[x_t^{[t]}(w)]_- = x_i(w)+ \sum_{s< t+1} a_{si}[\xi_s(w)]_-
\end{align*}
as desired.
For the first equation, we used $\sgn(b^{[t]}_{ti}) = +$ which follows from $t$ being a source in $Q^{[t]}$ (see the proof of \cref{lem:tau_is_mut_loop}) and now $t \leq i$.
Hence \eqref{eq:induction_2} holds for all $0 \leq t \leq i$. By setting $t=i$, we get 
\begin{align*}
    x_i^{[i]}(w)=x_i(w)+ \sum_{s< i} a_{si}[\xi_s(w)]_- = x_i(w)+ \sum_{s\precov i} a_{si}[\xi_s(w)]_- = \xi_i(w).
\end{align*}
Thus the assertion is proved. 
\end{proof}

Therefore, we see that the functions $\xi_i(w)$ control the sign $\bep_\gamma$ (in the sense of \cref{d:sign}) of the path $\gamma_\pi$:
\begin{align}\label{eq:sign}
    \bep_{\gamma_\pi}(w) = (\sgn(\xi_0(w)), \sgn(\xi_1(w)), \dots, \sgn(\xi_{N-1}(w))).
\end{align}

\begin{prop}\label{lem:action_tau}
The mutation loop $\tau$ satisfies
\begin{align}\label{eq:tau_formula}
    x_i(\tau(w)) = 
    -\xi_{i}(w) + \sum_{j \succov i} a_{ij}[\xi_{j}(w)]_-
\end{align}
for all $i \in I$ and $w \in \cX_\bs(\bR^\trop)$.
\end{prop}

\begin{proof}
Fixing $i \in I$, we claim that 
\begin{align}\label{eq:induction_3}
    x_i^{[t]}(w)= -\xi_i(w) + \sum_{i< s < t} a_{is}[\xi_s(w)]_-
\end{align}
holds for $i+1 \leq t\leq N$.
For $t=i+1$, we have $x_i^{[i+1]}(w)=-\xi_i(w)$ by \cref{lem:xi}. Assume that \eqref{eq:induction_3} holds for some $i+1 \leq t < N$.
In the quiver $Q^{[t]}$, we have $x_t^{[t]}(w)=\xi_t(w)$ from \cref{lem:xi}.
Then by the mutation formula, we get
\begin{align*}
    x_i^{[t+1]}(w)= x_i^{[t]}(w) + a_{it}[x_t^{[t]}(w)]_- = -\xi_i(w) + \sum_{i< s < t+1} a_{is}[\xi_s(w)]_-
\end{align*}
as desired. Hence \eqref{eq:induction_3} holds for all $i+1\leq t \leq N$. By setting $t=N$, we get the asserted equation \eqref{eq:tau_formula}.
\end{proof}

\begin{cor}\label{cor:clusterDT}
The mutation loop $\tau$ is the cluster DT transformation.
Namely, it satisfies $x_i(\tau(w))=-x_i(w)$ for all $w \in \cC_{(v_0)}^+$.
\end{cor}

\begin{proof}
We $x_i(w) \geq 0$ for all $i \in I$ for all $w \in \cC^+_{(v_0)}$. From the defining relation \eqref{eq:xi_definition}, we inductively get $\xi_i(w)=x_i(w) \geq 0$ for all $i \in I$. Then \cref{lem:action_tau} tells us that $x_i(\tau(w))=-x_i(w)$ for all $i \in I$. 
\end{proof}
By the uniqueness of the cluster DT transformation (\cite[Theorem 3.2]{GS18}), it follows that the mutation loop $\tau$ does not depend on the choice of an admissible labeling $\pi$. In other words, an admissible labeling $\pi$ determines a representation path $\gamma_\pi$ of $\tau$.

\subsection{Coxeter matrix}
We are going to see that the Coxeter matrix in the sense of \cite{ASS} appears as a presentation matrix of $\tau$.

\begin{dfn}
Let $Q$ be an acyclic quiver.
\begin{enumerate}
    \item Define the matrix $M=(m_{ij})_{i,j \in I}$ by 
    \begin{align*}
        m_{ij}:= \#\{ \text{directed paths in $Q$ from $i$ to $j$} \}
    \end{align*}
    for $i,j \in I$.\footnote{It is the transposition of the ``$M$-matrix" in \cite{Pena}.} Note that $m_{ii}=1$ for all $i \in I$.
    Observe that $m_{ij} >0$ only if $i \precceq j$. 
    \item The matrix $\Phi=(\phi_{ij})_{i,j \in I}:=-M^{-1} \cdot M^{\tr}$ is called the \emph{Coxeter matrix}, following \cite{ASS}.
\end{enumerate} 
\end{dfn}
See \cref{ex:D4_sign} for examples of $\Phi$. 

\begin{lem}\label{lem:A*M}
We have $\sum_{k \in I}a_{ik}m_{kj}=(1-\delta_{ij})m_{ij} =\sum_{k \in I}m_{ik}a_{kj}$ for $j,k \in I$.
\end{lem}

\begin{proof}
Observe that the left-hand side counts all the directed paths from $i$ to $j$ that pass through some vertex $k \succov i$. If $i \neq j$, then every directed path from $i$ to $j$ must pass through $I_{\succov i}$, so that $\sum_{k \in I}a_{ik}m_{kj}=m_{ij}$ holds. If $i=j$, then such a directed path does not exist by the acyclicity. Hence we have the first equality. The proof for the second equality is similar.
\end{proof}

\begin{lem}\label{lem:phi}
We have $\phi_{ij} = -m_{ji}+\sum_{k \in I}a_{ik}m_{jk}$ for $i,j \in I$. 
\end{lem}

\begin{proof}
First note that the matrix $\mathrm{Id}_N - A$
is inverse to $M$. 
Indeed, 
\begin{align*}
    \sum_{k \in I} (\delta_{ik} - a_{ik})m_{kj}
    = m_{ij} - \sum_{k \neq i} a_{ik}m_{kj}
    = m_{ij} - (1-\delta_{ij}) m_{ij}
    = \delta_{ij}m_{ii}
    = \delta_{ij}.
\end{align*}
Then we get
\begin{align*}
    \phi_{ij} = -\sum_{k \in I} (-\delta_{ik} + a_{ik}) m_{jk} = -
    m_{ji} + \sum_{i \neq k} a_{ik} m_{jk},
\end{align*}
as asserted. 
\end{proof}
We have the following important lemma:

\begin{lem}[Proof in \cref{sec:proofs}]\label{lem:M^tr*Phi}
If the quiver $Q$ is representation infinite, then each entry of the matrix $M^\tr \Phi^n$ is non-negative for any $n \geq 0$.
\end{lem}

For $n \geq 0$, we denote the $(j,k)$-entry of $\Phi^n$ by $\phi_{jk}^{(n)}$. The above lemma says that $\sum_{j \in I} m_{ji}\phi_{jk}^{(n)} \geq 0$.

\begin{lem}\label{lem:xi_M}
If the quiver $Q$ is representation infinite, then we have the following for $n \geq 0$, $i \in I$ and $w \in \cC^-_{(v_0)}$:
\begin{enumerate}
\renewcommand{\labelenumi}{(\arabic{enumi})$_n$}
    \item $\xi_i(\tau^n(w)) \leq 0$.
    \item $\xi_i(\tau^n(w)) = \sum_{\ell \in I} m_{\ell i} x_\ell(\tau^n(w))$.
    \item $x_i(\tau^{n+1}(w)) = \sum_{k \in I} \phi_{ik} x_k(\tau^n(w))$.
\end{enumerate}
\end{lem}

\begin{proof}
We prove it by induction on $n$.
First, we verify (1)$_0$ and (2)$_0$.
The statement (1)$_0$ follows from
\begin{align*}
    \xi_i(w) = x_i(w) + \sum_{k \precov i} a_{ki} [\xi_k(w)]_-
    \leq x_i(w) \leq 0,
\end{align*}
since $[a]_- \leq 0$ for $a \in \bR$.
For (2)$_0$, we proceed by induction along the poset $(I, \precceq)$.
It is clear for a minimal vertex $s$, since $m_{\ell s}=\delta_{\ell s}$ in this case.
For a general vertex $i$,
\begin{align*}
\xi_i(w) &= x_i(w) + \sum_{k \precov i} a_{ki} [\xi_k(w)]_-\\
&= x_i(w) + \sum_{k \precov i} a_{ki} \sum_{\ell \in I} m_{\ell k} x_\ell(w) & (\text{\text{(1)$_0$} \& induction assumption})&\\
&= m_{ii} x_i(w) + \sum_{\ell \in I} \big( (1 - \delta_{\ell i}) m_{\ell i} \big) x_\ell(w) & (\text{\cref{lem:A*M}})&\\
&= \sum_{\ell \in I} m_{\ell i} x_\ell(w).
\end{align*}
Thus (2)$_0$ is proved. 

Next, we prove $(1)_n \,\&\, (2)_n \Rightarrow (3)_n$ for all $n \geq 0$:
\begin{align*}
    x_i(\tau^{n+1}(w)) &= -\xi_{i}(\tau^n(w)) + \sum_{j \succov i} a_{ij}[\xi_{j}(\tau^n(w))]_- & (\text{\cref{lem:action_tau}})& \\
    &=  -\sum_{\ell \in I} m_{\ell i} x_\ell(\tau^n(w)) + \sum_{j \succov i} a_{ij} \sum_{\ell \in I} m_{\ell j} x_\ell(\tau^n(w)) & (\text{(1)$_n$ \& (2)$_n$})&\\
    &= \sum_{\ell \in I} \phi_{i \ell} x_\ell(\tau^n(w)). & (\text{\cref{lem:phi}})&
\end{align*}
In particular, (3)$_0$ is proved.

Finally, we prove $(3)_0$--$(3)_n \Rightarrow (1)_{n+1} \,\&\, (2)_{n+1}$ for all $n \geq 0$.
We proceed by induction along $(I,\precceq)$.
For a minimal vertex $s$ in $(I, \precceq)$, (2)$_{n+1}$ clearly holds and we get
\begin{align*}
    \xi_s(\tau^{n+1}(w)) = x_s(\tau^{n+1}(w)) = \sum_{k} \phi^{(n+1)}_{sk} x_k(w)
\end{align*}
by $(3)_0$--$(3)_n$.
Also, $(M^\tr \Phi^n)_{sk} = \phi^{(n)}_{sk}$ since $s$ is minimal.
Therefore, $(1)_{n+1}$ for $s$ follows from \cref{lem:M^tr*Phi}.
For a general vertex $i \in I$, 
\begin{align*}
    \xi_i(\tau^{n+1}(w)) &= x_i(\tau^{n+1}(w)) + \sum_{k \precov i} a_{ki} [\xi_k(\tau^{n+1}(w))]_-\\
    &= m_{ii} x_i(\tau^{n+1}(w)) + \sum_{k} a_{ki} \sum_\ell m_{\ell k} x_\ell(\tau^{n+1}(w)) & (\text{induction assumption})&\\
    &= \sum_{\ell} m_{\ell i} x_\ell(\tau^{n+1}(w)). & (\text{\cref{lem:A*M}})&
\end{align*}
Here, the second equation follows from the induction assumption on $i$ for $(2)_{n+1}$ and $(1)_{n+1}$ since $k \prec i$.
Thus we obtain $(2)_{n+1}$.
By this formula and $(3)_0$--$(3)_n$, we also get
\begin{align*}
    \xi_i(\tau^{n+1}(w)) = \sum_{\ell} m_{\ell i} x_\ell(\tau^{n+1}(w)) = \sum_{\ell, k} m_{\ell i} \phi^{(n+1)}_{\ell k} x_k(w).
\end{align*}
Since $w \in \cC^-_{(v_0)}$, we obtain $(1)_{n+1}$ by \cref{lem:M^tr*Phi}.
\end{proof}

Recall that the mutation loop $\tau$ first maps $\cC^+_{(v_0)}$ to $\cC^-_{(v_0)}$ by \cref{cor:clusterDT}. 
Then by combining \eqref{eq:sign} and \cref{lem:xi_M}, we get the following.

\begin{thm}\label{cor:SS_infinite}
If $Q$ is representation infinite, then $\gamma_\pi$ is sign-stable on $\cC^+_{(v_0)} \cup \cC^-_{(v_0)}$ with the stable sign $\boldsymbol{\epsilon}_{\gamma_\pi}^\stab = \underbrace{(-, \cdots, -)}_N$ and the stable presentation matrix $\Phi$ for any admissible labeling $\pi$ of $Q$.
\end{thm}

\subsection{Inverse of the cluster Donaldson--Thomas transformation}
As a slight digression, let us investigate the sign stability of the ``reversed path'' of $\gamma_\pi$:
\begin{align*}
    \gamma^{-1}_\pi : v_N \overbar{N-1} v_{N-1} \overbar{N-2} \cdots \overbar{1} v_{1} \overbar{0} v_{0}.
\end{align*}
It represents the mutation loop $\tau^{-1} \in \Gamma_\bs$. 
In general, it is non-trivial whether the sign stability of a path $\gamma$ implies that for the reversed path $\gamma^{-1}$.
A reader interested only in the main results may safely skip this subsection. 

During this subsection, we write $x_i(w):=x_i^{[N]}(w)$ for $i\in I$ and $w \in \X_\bs(\bR^\trop)$.
Let us define the PL functions $\eta_i(w)$ on $\cX_\bs(\bR^\trop)$, $i \in I$ by $\eta_i(w) := x_{N-i}(w)$ if $N-i$ is maximal in $(I, \precceq)$ and the recursive relation
\begin{align}
    \eta_i(w) = x_{N-i}(w) + \sum_{j \in I} a_{N-i, N-j} [\eta_j(w)]_+
\end{align}
along $(I, \precceq)$.
By comparing the definitions of $\eta_i(w)$ and $\xi_i(w)$, one can prove the following properties around $\eta_i(w)$:

\begin{lem}[Corresponding to \cref{lem:xi}]\label{lem:eta}
We have $x^{[N-i]}_{N-i}(w) = \eta_i(w)$ for all $i\in I$ and $w \in \cX_\bs(\bR^\trop)$.
\end{lem}

Therefore, we have
\begin{align}\label{eq:sign_rev}
    \bep_{\gamma^{-1}_\pi}(w) = (\sgn(\eta_0(w)), \sgn(\eta_1(w)), \dots, \sgn(\eta_{N-1}(w)))
\end{align}
in contrast to \eqref{eq:sign}.

\begin{lem}[Corresponding to \cref{lem:action_tau}]\label{lem:action_tau^-1}
The mutation loop $\tau^{-1}$ satisfies
\begin{align}\label{eq:tau_formula_inv}
    x_i(\tau^{-1}(w)) = 
    -\eta_{N-i}(w) + \sum_{j \in I} a_{ji}[\eta_{N-j}(w)]_+
\end{align}
for all $i \in I$ and $w \in \cX_\bs(\bR^\trop)$.
\end{lem}

Thus we obtain the following\footnote{
We can get \cref{cor:clusterDT_inv} from the fact that the cluster modular group $\Gamma_\bs$ acts on the fan $\{ \text{the faces of the cone } \cC^+_{(v)} \mid v \in \bExch_\bs \}$ simplicially.
}:
\begin{cor}[Corresponding to \cref{cor:clusterDT}]\label{cor:clusterDT_inv}
The mutation loop $\tau^{-1}$ satisfies $x_i(\tau^{-1}(w))=-x_i(w)$ for all $w \in \cC_{(v_0)}^-$.
\end{cor}

\begin{lem}[Corresponding to \cref{lem:M^tr*Phi}]\label{lem:M*Phi^-1}
If the quiver $Q$ is representation infinite, then each entry of the matrix $M \Phi^{-n}$ is non-negative for any $n \geq 0$.
\end{lem}
This lemma is proved in \cref{sec:proofs} together with \cref{lem:M^tr*Phi}.
By using these lemmas, one can prove the following by a discussion parallel to the proof of \cref{lem:xi_M}.

\begin{lem}
If the quiver $Q$ is representation infinite, then we have the following for $n \geq 0$, $i \in I$ and $w \in \cC^+_{(v_0)}$:
\begin{enumerate}
\renewcommand{\labelenumi}{(\arabic{enumi})$_n$}
    \item $\eta_{N-i}(\tau^{-n}(w)) \geq 0$.
    \item $\eta_{N-i}(\tau^{-n}(w)) = \sum_{\ell \in I} m_{i \ell} x_\ell(\tau^{-n}(w))$.
    \item $x_i(\tau^{-n-1}(w)) = \sum_{k \in I} \psi_{ik} x_k(\tau^{-n}(w))$.
\end{enumerate}
Here, $\psi_{ik}$ denotes the $(i,k)$-component of $\Phi^{-1}$.
\end{lem}

Therefore, we get the following:
\begin{thm}\label{cor:SS_infinite_inv}
If $Q$ is representation infinite, then $\gamma_\pi^{-1}$ is sign-stable on $\cC^+_{(v_0)} \cup \cC^-_{(v_0)}$ with the stable sign $\boldsymbol{\epsilon}_{\gamma_\pi^{-1}}^\stab = \underbrace{(+, \cdots, +)}_N$ and the stable presentation matrix $\Phi^{-1}$ for any admissible labeling $\pi$ of $Q$.
\end{thm}

\begin{rmk}
As we mentioned at the beginning of this subsection, it seems to be non-trivial whether the sign stability of a path $\gamma$ implies that for the reversed path $\gamma^{-1}$.
However, we do not know any counter-example.

From the results \cref{cor:clusterDT,cor:clusterDT_inv}, we are tempted to conjecture that if $\gamma$ is sign-stable with the stable sign $\bep^\mathrm{stab}_\gamma = (\epsilon_0, \dots, \epsilon_{N-1})$, then the reverse $\gamma^{-1}$ is also sign-stable with the stable sign $\bep^\mathrm{stab}_{\gamma^{-1}} = (-\epsilon_{N-1}, \dots, -\epsilon_0)$.
There is an example of a path $\gamma$ such that both $\gamma$ and $\gamma^{-1}$
are sign-stable but their stable signs are not related as above.
\end{rmk}
\section{Finite-tame-wild trichotomy}\label{sec:trich}

We are ready to prove our main theorem:
\begin{thm}\label{thm:trich}
Let $\pi$ be an admissible labeling of an acyclic quiver $Q$.
Then,
\begin{enumerate}
    \item $Q$ is representation finite if and only if $\gamma_\pi$ is not basic sign-stable.
    \item $Q$ is tame if and only if $\gamma_\pi$ is basic sign-stable and $\lambda_\tau=1$.
    \item $Q$ is wild if and only if $\gamma_\pi$ is basic sign-stable and $\lambda_\tau>1$.
\end{enumerate}
Here $\lambda_\tau=\lambda_\tau^{(v_0)}$ denotes the cluster stretch factor (\cref{d:cluster stretch factor}) of the mutation loop $\tau$. 
\end{thm}

\begin{proof}
Let $\gamma := \gamma_\pi$.
Let us assume that $Q$ is representation finite.
Then, $\tau$ is of finite order (see \cite{ASS12} for instance) in $\Gamma_\bs$.
Let $r$ be the order of $\tau$.
Then, for $w \in \cC^+_{(v_0)}$, 
\begin{align*}
    \bep_\gamma(w) = (+, \dots, +) \quad \text{and} \quad
    \bep_\gamma(\tau(w)) = (-, \dots, -)
\end{align*}
by the proofs of \cref{cor:clusterDT} and \cref{lem:xi_M} (1)$_0$.
Therefore, we have
\begin{align*}
    \bep_\gamma(\tau^{nr}(w)) = (+, \dots, +) \quad \text{and} \quad
    \bep_\gamma(\tau^{nr+1}(w)) = (-, \dots, -)
\end{align*}
for any $n \geq 0$ and $w \in \cC^+_{(v_0)}$ by the periodicity of $\tau$, which implies that the path $\gamma$ is not basic sign-stable.

Next, we assume that $Q$ is representation infinite.
Then, $\gamma$ is basic sign-stable by \cref{cor:SS_infinite}.
Therefore, (1) is proved.
Moreover, the cluster stretch factor $\lambda_\tau$ is exactly the spectral radius of the Coxeter matrix $\Phi$, since it gives the stable presentation matrix of $\tau$.
If $Q$ is tame, an explicit formula of the characteristic polynomial of $\Phi$ is known \cite[Section 1.5]{Pena}.
By this formula, one can deduce $\lambda_\tau= 1$.
On the other hand, if $Q$ is wild, it is known that $\lambda_\tau >1$ \cite{ACam,Rin94}.
Thus, we get (2) and (3).
\end{proof}

\begin{ex}\label{ex:D4_sign}
Let $Q$ be the quiver with an admissible labeling $\pi$ shown in the left of \cref{fig:D4}.
It is of type $D_4$, which is representation finite. Its Coxeter matrix is
\begin{align*}
    \Phi = 
    \begin{pmatrix}
        0 & 1 & 0 & 0 \\
        1 & 1 & 1 & 1 \\
        -1 & -1 & -1 & 0 \\
        -1 & -1 & 0 & -1
    \end{pmatrix}.
\end{align*}
The admissible sequence $\gamma_\pi$  is not basic sign-stable by \cref{thm:trich}.
Indeed, the sign of $\gamma=\gamma_\pi$ at an initial point $w\in \interior \cC_{(v_0)}^+$ evolves as shown in the left column of \cref{tab:D4_sign}. One can easily see its 4-periodicity, which is unstable. 
The Coxeter matrix $\Phi$ is equal to the presentation matrix $E_{\gamma}^{(-,-,-,-)}$.
One can verify that $\Phi^3 = - \mathrm{Id}$, hence $\Phi$ does not have the Perron--Frobenius property.

Next, adding an arrow to $Q$ as shown in the right of \cref{fig:D4}, we obtain a wild acyclic quiver $Q'$. Take an admissible labeling $\pi'$ in the same way as $\pi$. 
Then by \cref{thm:trich}, the admissible sequence $\gamma'=\gamma_{\pi'}$ is basic sign-stable with the stable presentation matrix (= Coxeter matrix)
\begin{align*}
    \Phi' = \begin{pmatrix}
        0 & 1 & 0 & 0\\
        2 & 2 & 2 & 1\\
        1 & 1 & 0 & 1\\
        -2& -2& -1& -1 
    \end{pmatrix}.
\end{align*}
Its characteristic polynomial is
\begin{align*}
    x^4 - x^3 -3x^2 -x +1,
\end{align*}
and the cluster stretch factor of $\tau$ (= spectral radius of $\Phi'$) is given by its largest root
\begin{align*}
    \lambda_\tau = \frac{1}{4}\Big({\textstyle \sqrt{21} + \sqrt{2 \sqrt{21} + 6} +1 }\Big) = 2.369205407092467...
\end{align*}

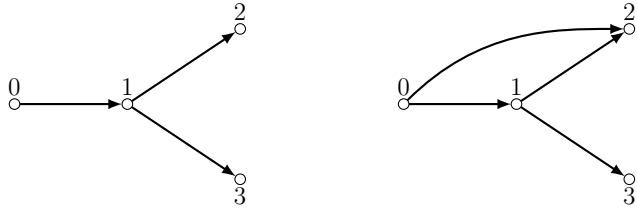
\begin{figure}[h]
    \centering
    \begin{tikzpicture}
    \draw(0,0) circle(2pt) node[above,scale=0.8]{$0$};
    \draw(1.5,0) circle(2pt) node[above,scale=0.8]{$1$};
    \draw(3,1) circle(2pt) node[above,scale=0.8]{$2$};
    \draw(3,-1) circle(2pt) node[below,scale=0.8]{$3$};
    \qarrowop{1.5,0}{0,0};
    \qarrowop{3,1}{1.5,0};
    \qarrowop{3,-1}{1.5,0};
    \end{tikzpicture}
    \qquad \qquad
    \begin{tikzpicture}
    \draw(0,0) circle(2pt) node[above,scale=0.8]{$0$};
    \draw(1.5,0) circle(2pt) node[above,scale=0.8]{$1$};
    \draw(3,1) circle(2pt) node[above,scale=0.8]{$2$};
    \draw(3,-1) circle(2pt) node[below,scale=0.8]{$3$};
    \qarrowop{1.5,0}{0,0};
    \qarrowop{3,1}{1.5,0};
    \qarrowop{3,-1}{1.5,0};
    \draw [->, >=latex, shorten >=2pt,shorten <=2pt, thick](0,0) .. controls (1,1) and (2,1) .. (3,1);
    \end{tikzpicture}
    \caption{Left: a representation finite quiver $Q$ (type $D_4$) with an admissible labeling $\pi$. Right: a wild quiver $Q'$ with an admissible labeling $\pi'$.}
    \label{fig:D4}
\end{figure}

\begin{table}[h]
    \centering
    \begin{tabular}{c|c|c}
    $n$ & $\bep_{\gamma}(\tau^n(w))$ & $\bep_{\gamma'}(\tau^n(w'))$\\
    \hline
    0 & $(+,+,+,+)$ & $(+,+,+,+)$\\
    1 & $(-,-,-,-)$ & $(-,-,-,-)$\\
    2 & $(-,-,-,-)$ & $(-,-,-,-)$\\
    3 & $(-,-,-,-)$ & $(-,-,-,-)$\\
    4 & $(+,+,+,+)$ & $(-,-,-,-)$\\
    5 & $(-,-,-,-)$ & $(-,-,-,-)$\\
    6 & $(-,-,-,-)$ & $(-,-,-,-)$\\
    7 & $(-,-,-,-)$ & $(-,-,-,-)$\\
    8 & $(+,+,+,+)$ & $(-,-,-,-)$\\
    9 & $(-,-,-,-)$ & $(-,-,-,-)$\\
    10 & $(-,-,-,-)$ & $(-,-,-,-)$\\
       & $\vdots$ & $\vdots$
    \end{tabular}
    \caption{The signs for $\gamma = \gamma_\pi$ and $\gamma' = \gamma_{\pi'}$ at $\tau^n(w)$ and $\tau^n(w')$ for $w \in \interior \cC^+_{(v_0)}$ and $w' \in \interior \cC^+_{(v'_0)}$ with $Q^{(v_0)} = Q$ and $Q^{(v'_0)} = Q'$, respectively.}
    \label{tab:D4_sign}
\end{table}
\end{ex}
\section{Dynamics of the cluster Donaldson--Thomas transformations}\label{sec:dyn}
Here, we study some dynamical properties of the cluster DT transformation $\tau$.

\subsection{Entropy}\label{subsec:entropy}
In this subsection, we compute several kinds of entropies of the cluster Donaldson--Thomas transformation $\tau$ of an acyclic quiver $Q$.
First, we recall some results on entropies of mutation loops.

\begin{defi}
We say that an invertible matrix $M$ satisfies the \emph{palindromicity property} if the characteristic polynomials of $M$ and $(M^{-1})^{\tr}$ are the same up to the overall sign.
\end{defi}

We conjectured that the presentation matrix $E_\gamma^{\bep_\gamma(w)}$ of a representation path $\gamma$ of a mutation loop at any point $w \in \cX_\bs(\bR^\trop)$ satisfies the palindromicity property \cite[Conjecture 3.13]{IK19}.
It still remains open and there is no counter-example so far.

For a mutation loop $\phi \in \Gamma_\bs$, let us consider its actions
\begin{align*}
    \phi_a: \cA_\bs \to \cA_\bs \quad \text{and} \quad \phi_x: \cX_\bs \to \cX_\bs
\end{align*}
on the cluster varieties $\cA_\bs$ and $\cX_\bs$, respectively. In each case, the restriction of the action to a cluster chart gives a birational map on $(\bC^\ast)^N$, so we can consider its \emph{algebraic entropy} \cite{BV99}. Since the algebraic entropy is invariant under conjugations by birational maps, the algebraic entropies $h_\mathrm{alg}(\phi_a)$ and $h_\mathrm{alg}(\phi_x)$ are defined independently of the choices of cluster charts.   
See \cite[Section 2.4]{IK19} for a detail. 

\begin{thm}[{\cite[Corollary 1.2]{IK19}}]
Let $\phi$ be a mutation loop with a basic sign-stable representation path $\gamma$ starting at $v_0$. Assume that the stable presentation matrix of $\phi$ satisfies the palindromicity property.
Then, we have
\begin{align*}
    h_\mathrm{alg}(\phi_a) = h_\mathrm{alg}(\phi_x) = \log \lambda_\phi^{(v_0)}.
\end{align*}
\end{thm}
When we drop the palindromicity assumption, we still have an estimate between the algebraic entropies and the cluster stretch factor (\cite[Theorem 1.1]{IK19}). 

Let us denote by $A_{Q,W}$ the Ginzburg dg algebra of the quiver $Q$ with a nondegenerate potential $W$, and denote by $\mathsf{D}(A_{Q,W})$, $\mathsf{D}_\mathsf{fd}(A_{Q,W})$ and $\mathsf{per}(A_{Q,W})$ the derived category, the finite-dimensional derived category and the perfect derived category of $A_{Q,W}$, respectively.

\begin{thm}[{\cite[Theorem 1.1]{Kan21}}]
Let $\phi$ be a mutation loop with a basic sign-stable representation path $\gamma$ starting from $v_0$. Assume that the stable presentation matrix of $\phi$ satisfies the palindromicity property. 
Then, there is a derived autoequivalence $F_\phi: \mathsf{D}(A_{Q,W}) \to \mathsf{D}(A_{Q,W})$ such that it is restricted to the subcategories $F_\phi|_{\mathsf{D}_\mathsf{fd}}: \mathsf{D}_\mathsf{fd}(A_{Q,W}) \to \mathsf{D}_\mathsf{fd}(A_{Q,W})$ and $F_\phi|_{\mathsf{per}}: \mathsf{per}(A_{Q,W}) \to \mathsf{per}(A_{Q,W})$ and
\begin{align*}
    h_T(F_\phi|_{\mathsf{D}_\mathsf{fd}}) = h_0(F_\phi|_{\mathsf{per}}) = \log \lambda_\phi^{(v_0)}
\end{align*}
Here, $\{h_T\}_{T \in \bR}$ denotes the categorical entropy.
\end{thm}

\begin{lem}\label{lem:palind_Coxeter}
The Coxeter matrix $\Phi$ of an acyclic quiver satisfies the palindromicity property.
\end{lem}
\begin{proof}
It is obvious from the equation
\begin{align}\label{eq:Phi_check}
    M^\tr \Phi^{-1} (M^{-1})^{\tr} = \Phi^\tr.
\end{align}
\end{proof}

Therefore, we can compute these kinds of entropies associated with $\tau$ and $\tau^{-1}$ by \cref{cor:clusterDT,cor:clusterDT_inv}:
\begin{thm}\label{thm:ent_tau}
Let $Q$ be a representation infinite acyclic quiver.
Then, we have
\begin{align*}
    h_\mathrm{alg}(\tau_a^{\pm 1}) = h_\mathrm{alg}(\tau_x^{\pm 1}) = h_T(F_{\tau^{\pm 1}}|_{\mathsf{D}_\mathsf{fd}}) = h_0(F_{\tau^{\pm 1}}|_{\mathsf{per}}) = \log \rho(\Phi).
\end{align*}
Here, $\rho(\Phi)$ denotes the spectral radius of the Coxeter matrix $\Phi$ of $Q$.
\end{thm}
Here, we note that $\rho(\Phi) = \rho(\Phi^{-1})$ by \eqref{eq:Phi_check}.

\smallskip
By using the presentation of the cluster modular group of $\Gamma_\bs$ given in \cite[Section 3.3]{ASS12}, we can conclude that the algebraic entropy of any mutation loops of representation finite or tame acyclic quiver is zero:
\begin{thm}\label{thm:ent_vanish_tame}
Let $Q$ be a representation finite or tame acyclic quiver.
Then, we have
\begin{align*}
    h_\mathrm{alg}(\phi_a) = h_\mathrm{alg}(\phi_x) = 0
\end{align*}
for any $\phi \in \Gamma_\bs$.
\end{thm}

\begin{proof}
If $Q$ is representation finite, then every mutation loop is of finite order, so the algebraic entropy is zero.
We assume that $Q$ is tame.
We show that for each $\phi \in \Gamma_\bs$, there is $r, t \in \bZ$ such that $\phi^r = \tau^t$.
If it is proved, then we get the statement since $h_\mathrm{alg}(\tau^n_z) = n \cdot h_\mathrm{alg}(\tau_z)$ by the property of the algebraic entropy (for $n \in \bZ_{\geq 0}$) and \cref{thm:ent_tau}. 
For the types $\tilde{D}_4$, $\tilde{E}_6$, $\tilde{E}_7$ and $\tilde{E}_8$, it is obvious since these are in the form of $\langle \tau \rangle \times H$ with finite groups $H$ \cite[Table 1]{ASS12}.

For the type $\tilde{A}_{p, q}$ with $p \neq q$,
one can express $\phi = \tau^a r_1^{k_1} r_2^{k_2}$ for some $a, k_1, k_2 \in \bZ$.
Then, $\phi^{p+q} = \tau^{(a-k_2)p + (a-k_1)q}$ since we have $r_1^{p+q} = \tau^q$ and $r_2^{p+q} = \tau^p$ by the relations in $\Gamma_\bs$.
The case $p=q$ is proved in a similar way.

For the type $\tilde{D}_{n-1}$ with $n>5$, 
one can express $\phi = \tau^a \sigma^b \rho_1^{g_1} \rho_n^{g_n}$ for some $a, b, g_1, g_2 \in \bZ$.
Then, $\phi^2 = \tau^{2a + b(n-3)}$ by the relations in $\Gamma_\bs$.
\end{proof}

\subsection{North-South type dynamics}\label{subsec:North_dyn}
In this subsection, we see a particular dynamical property of the action $\tau: \cX_\bs(\bR^\trop) \to \cX_\bs(\bR^\trop)$.

For an $\bR_{>0}$-invariant subset $\Omega \subset \cX_\bs(\bR^\trop)$, we write $\bS \Omega := (\Omega \setminus \{0\}) / \bR_{>0}$.
Let $[w]\in \bS \Omega$ denote the point represented by $w \in \Omega$.

\begin{thm}
Let $Q$ be a wild acyclic quiver and $\bs := \bs_Q$.
Then, there exist two points $p^\pm_{\tau} \in \bS \cX_\bs(\bR^\trop)$ such that
\begin{align}\label{eq:limit_tau^n}
    \lim_{n \to \infty} [\tau^{\pm n}(w)] = p_\tau^\pm
\end{align}
for all $w \in \cC^+_{(v_0)} \cup \cC^-_{(v_0)}$.
\end{thm}

\begin{proof}
We only prove that $[\tau^n(w)] \xrightarrow{n \to \infty} p^+_\tau$.
The other is proved in the same way.
Since $\tau$ maps $\cC^+_{(v_0)}$ to $\cC^-_{(v_0)}$, it suffices to consider the orbit of a point $w \in \cC^-_{(v_0)}$.

Let $\bx_+ = (x_i)_{i \in I} \in \bR^I$ be an eigenvector of the spectral radius $\rho = \rho(\Phi)$ so that $\Phi\bx_+=\rho \bx_+$, which is unique up to multiplication by the Perron--Frobenius property.
Then, we will prove that there exists $\lambda_+ =\lambda_+(w)\neq 0$ such that
\begin{align*}
    \lim_{n \to \infty} \frac{1}{\rho^n} \bx^{(v_0)}(\tau^n(w)) = \lambda_+ \bx_+.
\end{align*}
Let $D := \mathrm{diag}(\bx_+)$ and $P := D^{-1} \Phi D / \rho$.
Then, $P^\infty := \lim_{n \to 0} P^n$ exists and each column of it is a multiple of the vector $(1, \dots, 1)^\tr$ (see \cite[Proof of Theorem 3.5]{Pena}).
Thus,
\begin{align*}
    \frac{1}{\rho^n} \bx^{(v_0)}(\tau^n(w)) 
    =\frac{1}{\rho^n} \Phi^n \bx^{(v_0)}(w)
    =D P^n D^{-1}\bx^{(v_0)}(w)
    \xrightarrow{n \to 0}
    D P^\infty D^{-1}\bx^{(v_0)}(w) = \lambda_+ \bx_+
\end{align*}
where $\lambda_+ = \sum_{i \in I} x_i^{-1} p_i x_i^{(v_0)}(w)$ with $p_i$ the constant value of the $i$th column of $P^\infty$.

It is known that there is $\mathbf{y}_+ \in \bR^I_{>0}$ such that $\Phi^\tr \mathbf{y}_+ = \rho(\Phi) \mathbf{y}_+$ \cite[Theorem 2.1]{Pena}.
Then, 
\begin{align*}
    \lambda_+ \langle \bx_+, \mathbf{y}_+ \rangle
    = \lim_{n \to \infty} \frac{1}{\rho^n} \langle \bx^{(v_0)}(\tau^n(w)) , \mathbf{y}_+ \rangle
    = \lim_{n \to \infty} \frac{1}{\rho^n} \langle \Phi^n \bx^{(v_0)}(w) , \mathbf{y}_+ \rangle
    = \langle \bx^{(v_0)}(w) , \mathbf{y}_+ \rangle <0
\end{align*}
since $w \in \cC^-_{(v_0)}$.
Therefore $\lambda_+ \neq 0$. Then by letting $p_\tau:=[\bx_+]$, we obtain $[\tau^n(w)] \to p_\tau$ in $\bS\X_\bs(\bR^\trop)$. 
\end{proof}

Combining with \cite[Theorem 1.3]{Kan21}, we get the following:

\begin{cor}
The functor $F_\phi|_{\mathsf{D}_{\mathsf{fd}}}$ is pseudo-Anosov with the stretch factor $\rho(\Phi)$ in the sense of \cite{FFHKL}.
\end{cor}

\begin{ex}\label{ex:D4_dyn}
We again consider the quivers $Q$, $Q'$ and its admissible labelings $\pi$, $\pi'$ in \cref{ex:D4_sign}.
Let $\bs$ and $\bs'$ be the mutation classes containing $Q$ and $Q'$ as the initial quivers, respectively.
We observe the action of $\tau$ on $\bS \cX_\bs(\bR^\trop)$ by tracing the orbit of $\ell \in \cX_\bs(\bR^\trop)$ and $\ell' \in \cX_{\bs'}(\bR^\trop)$ with $\mathbf{x}^{(v_0)}(\ell) = \mathbf{x}^{(v_0)}(\ell') = (1, 1, 1, 1)$ as shown in \cref{tab:D4_dyn}.
Here for $\mathbf{v} \in \bR^I$, we write $u(\mathbf{v}) := \mathbf{v} / \| \mathbf{v} \|$.

The orbit for the $D_4$-quiver $Q$ is again 4-periodic. 
For the wild quiver $Q'$, we may observe that the orbit approaches the unit right $\lambda_\tau$-eigenvector of the Coxeter matrix $\Phi'$, which is given by
\begin{align*}
    (-0.2947575153522..., -0.6983410991536..., -0.1513810480345..., 0.6344458169711...).
\end{align*}

\begin{table}[h]
    \centering
    \begin{tabular}{c|c|c}
    $n$ & $u(\mathbf{x}^{(v_0)}(\tau^n(\ell)))$ & $u(\mathbf{x}^{(v_0)}(\tau^n(\ell')))$ \\
    \hline
    $0$ & $(0.5, 0.5, 0.5, 0.5)$ & $(0.5, 0.5, 0.5, 0.5)$\\
    $1$ & $(-0.5, -0.5, -0.5, -0.5)$ & $(-0.5, -0.5, -0.5, -0.5)$\\
    $2$ & \small$(-0.1690..., -0.6761..., 0.5070..., 0.5070...)$ & \small$(-0.1025..., -0.7181..., -0.3077..., 0.6155...)$ \\
    $3$ & $(-0.8, 0.2, 0.4, 0.4)$ & \small $(-0.3201..., -0.7318..., -0.0914..., 0.5946...)$ \\
    $4$ & $(0.5, 0.5, 0.5, 0.5)$ & \small$(-0.2945..., -0.6812..., -0.1841..., 0.6444...)$\\
    $5$ & $(-0.5, -0.5, -0.5, -0.5)$ & \small$(-0.2877..., -0.7076..., -0.1399..., 0.6299...)$\\
    $6$ & \small$(-0.1690..., -0.6761..., 0.5070..., 0.5070...)$ & \small$(-0.2995..., -0.6946..., -0.1547..., 0.6354...)$\\
    & $\vdots$ & $\vdots$
    \end{tabular}
    \caption{The orbit of $\ell$ and $\ell'$ by the iteration of $\tau$.}
    \label{tab:D4_dyn}
\end{table}

\end{ex}
\section{Proof of \texorpdfstring{\cref{lem:M^tr*Phi,lem:M*Phi^-1}}{Lemmas 3.10 and 3.16}
}
\label{sec:proofs}
Let $k$ be an algebraically closed field.
We consider the path algebra $kQ$ and its right modules.
For $i \in I$, we denote by $S_i$ the simple module at $i$, and by $P_i$, $I_i$ the projective cover and the injective envelope of $S_i$, respectively.
Let $\tau$ denote the Auslander--Reiten translation for the $kQ$-modules. 
For a $kQ$-module $X$, let $\bdim X \in \bZ^N_{\geq 0}$ denote its dimension vector.

The following are the basic facts on the Coxeter matrix $\Phi$:

\begin{lem}
\begin{enumerate}
    \item For a non-projective $kQ$-module $X$, we have $(\bdim \tau X)^\tr = (\bdim X)^\tr \Phi$.
    \item For a non-injective $kQ$-module $X$, we have $(\bdim \tau^{-1} X)^\tr = (\bdim X)^\tr \Phi^{-1}$.
    \item The $i$th column vector of $M$ is equal to $\bdim I_i$.
    \item The $i$th row vector of $M$ is equal to $(\bdim P_i)^\tr$. 
\end{enumerate}
\end{lem}

See {\cite[Corollary IV.2.9 and Proposition III.3.8]{ASS}} for a proof. 


\begin{proof}[Proof of \cref{lem:M^tr*Phi,lem:M*Phi^-1}]
For $n \geq 0$, we have
\begin{align*}
    M^\tr \Phi^n 
    &= \begin{pmatrix}
        (\bdim I_0)^\tr\\
        \vdots\\
        (\bdim I_{N-1})^\tr
    \end{pmatrix}
    \Phi^n
    = \begin{pmatrix}
        (\bdim \tau^n I_0)^\tr\\
        \vdots \\
        (\bdim \tau^n I_{N-1})^\tr
    \end{pmatrix},\\
    M \Phi^{-n} 
    &= \begin{pmatrix}
        (\bdim P_0)^\tr\\
        \vdots\\
        (\bdim P_{N-1})^\tr
    \end{pmatrix}
    \Phi^{-n}
    = \begin{pmatrix}
        (\bdim \tau^{-n} P_0)^\tr\\
        \vdots \\
        (\bdim \tau^{-n} P_{N-1})^\tr
    \end{pmatrix}.
\end{align*}
The modules $\tau^n I_i$ (resp. $\tau^n P_i$) are non-trivial $kQ$-modules since $Q$ is representation infinite.
(These are called preinjective and preprojective modules in \cite[Chapter VIII]{ASS}, respectively.)
In particular, the dimension vectors $\bdim \tau^n I_i$ and $\bdim \tau^{-n} P_i$ are non-negative vectors and hence each entries of $M^\tr \Phi^n$ and $M \Phi^{-n}$ are non-negative.
\end{proof}

\end{document}